\documentclass[a4j,11pt,twoside]{article}
\usepackage{latexsym,amsmath,amssymb,amsthm}
\usepackage{amscd,mathrsfs}
\usepackage{fouriernc}
\usepackage{bm}
\usepackage[dvipdfmx]{graphicx}
\usepackage{fancybox}
\usepackage{ascmac}
\usepackage{multicol}
\usepackage{tabularx}
\usepackage[top=3.5cm,bottom=3.5cm,left=3cm,right=3cm]{geometry}
\theoremstyle{theorem}

\newtheorem{thm}{Theorem}[section]
\newtheorem{lem}{Lemma}[section]
\newtheorem{cor}{Corollary}[section]
\newtheorem{prop}{Proposition}[section]

\theoremstyle{definition}
\newtheorem{assump}{Assumption}[section]
\newtheorem{rem}{Remark}[section]

\newtheorem{prfLB}{{\it Proof of Theorem \ref{thm:LB}}\!\!}

\makeatletter
 
 \@addtoreset{equation}{section}
\makeatother
\title{Bounds and Gaps of Positive Eigenvalues of Magnetic Schr\"{o}dinger Operators with No or Robin Boundary Conditions}
\author{{\sc Norihiro Someyama}$^{*}$}
\date{{\small $^{*}$Shin-yo-ji Buddhist Temple, 5-44-4 Minamisenju, Arakawa-ku, Tokyo 116-0003 Japan}\\
{\small E-mail: {\tt philomatics@outlook.jp}}\\
{\small ORCID iD: https://orcid.org/0000-0001-7579-5352}}

\begin{document}
\maketitle
\markboth{N. Someyama}{Bounds and gaps of positive eigenvalues under RBC (2019.5)}

\begin{abstract}
We consider magnetic Schr\"{o}dinger operators on a bounded region $\Omega$ with the smooth boundary $\partial \Omega$ in Euclidean space ${\mathbb R}^d$. 
In reference to the result from Weyl's asymptotic law and P\'{o}lya's conjecture, P. Li and S. -T. Yau(1983) (resp. P. Kr\"{o}ger(1992)) found the lower (resp. upper) bound $\frac{d}{d+2}(2\pi)^2({\rm Vol}({\mathbb S}^{d-1}){\rm Vol}(\Omega))^{-2/d}k^{1+2/d}$ for the $k$-th (resp. ($k+1$)-th) eigenvalue of the Dirichlet (resp. Neumann) Laplacian.
We show in this paper that this bound relates to the upper bound for $k$-th excited state energy eigenvalues of magnetic Schr\"{o}dinger operators with the compact resolvent.
Moreover, we also investigate and mention the gap between two energies of particles on the magnetic field.
For that purpose, we extend the results by Li, Yau and Kr\"{o}ger to the magnetic cases with no or Robin boundary conditions on the basis of their ideas and proofs.
\end{abstract}

{\footnotesize
magnetic Schr\"{o}dinger operator, \ compact resolvent, \ Robin boundary condition
}

\section{Introduction}
In this paper, we define the set of all natural numbers as
\[
\mathbb N:=\{1,2,3,\ldots\}
\]
and denote the imaginary unit by $i:=\sqrt{-1}$.

We begin with a concise survey of the bounds for eigenavlues of Dirichlet or Neumann Laplacians.
Let $\Omega\subset {\mathbb R}^d$ be a bounded region with the smooth boundary $\partial \Omega$. 
We define the Dirichlet Laplacian $-\Delta^{\mathscr D}_{\Omega}$ (resp. Neumann Laplacian $-\Delta^{\mathscr N}_{\Omega}$) on $\Omega$ by the quadratic form
\[
q[u,v]:=\int_{\Omega}\nabla u\,\overline{\nabla v}\,dx
\]
on the form-domain $H_0^1(\Omega)$ (resp. $H^1(\Omega)$), where
\begin{align}
H^1(\Omega)&:=\{u\in L^2(\Omega):\nabla u\in L^2(\Omega)\}, \\
H_0^1(\Omega)&:=\{u\in H_0^1(\Omega):{\rm The\ support\ of}\ u,\ {\rm supp}(u),\ {\rm is\ compact\ in}\ \Omega\}.
\end{align}
Here $\nabla$ is the distributional gradient.
It is well known(e.g. \cite{RS}) that both the spectrum of $-\Delta^{\mathscr D}_{\Omega}$ and the spectrum of $-\Delta^{\mathscr N}_{\Omega}$ are discrete.
In addition, both eigenvalues of $-\Delta^{\mathscr D}_{\Omega}$ and eigenvalues of $-\Delta^{\mathscr N}_{\Omega}$ can be increasingly ordered as positive real numbers.

Suppose that $\mu$ denotes the eigenvalue of $-\Delta^{\mathscr D}_{\Omega}$ satisfying
\[
\left\{
\begin{array}{l}
 \displaystyle -\Delta\varphi=\mu\varphi\quad {\rm in}\ \,\Omega \vspace{2mm}\\
 \displaystyle \varphi\bigr|_{\partial \Omega}\equiv 0
\end{array}
\right.
\]
and that $\nu$ the eigenvalue of $-\Delta^{\mathscr N}_{\Omega}$ satisfying
\[
\left\{
\begin{array}{l}
 \displaystyle -\Delta\psi=\nu\psi\quad {\rm in}\ \,\Omega \vspace{2mm}\\
 \displaystyle \frac{\partial \psi}{\partial {\bf n}}\biggr|_{\partial \Omega}\equiv 0
\end{array}
\right.
\]
where $f\bigr|_S$ denotes the function $f$ whose domain is restricted to the region $S$, and $\partial/\partial {\bf n}$ the normal derivative at $\partial \Omega$:
\[
\frac{\partial \psi}{\partial {\bf n}}=\nabla \psi\cdot{\bf n}.
\]

We write ${\mathbb S}^{d-1}$ for the $d$-dimensional unit spherical surface.
P. Li and S. -T. Yau \cite{LY} proved that
\begin{align}
\sum_{j=1}^k\mu_j\ge \frac{d}{d+2}(2\pi)^2\left({\rm Vol}({\mathbb S}^{d-1}){\rm Vol}(\Omega)\right)^{-2/d}k^{1+2/d}
\label{eq:LY}
\end{align}
for any $k\in \mathbb N$, and, P. Kr\"{o}ger \cite{K2} proved that
\begin{align}
\sum_{j=1}^k\nu_j\le \frac{d}{d+2}(2\pi)^2\left({\rm Vol}({\mathbb S}^{d-1}){\rm Vol}(\Omega)\right)^{-2/d}k^{1+2/d}
\label{eq:Kroeger}
\end{align}
for any $k\in \mathbb N$, where ${\rm Vol}(E)$ denotes the volume of the set $E$.
We hereafter denote
\begin{align}
\label{eq:Cdk}
{\cal C}_{d,k}(\Omega):=\frac{d}{d+2}(2\pi)^2\left({\rm Vol}({\mathbb S}^{d-1}){\rm Vol}(\Omega)\right)^{-2/d}k^{1+2/d}.
\end{align}
In other words, it is known that the finite sum of eigenvalues of the Dirichlet Laplacian is larger than the finite sum of eigenvalues of the Neumann Laplacian, i.e., 
\begin{equation}
\label{eq:nCdkm}
\sum_{j=1}^k\nu_j\le {\cal C}_{d,k}(\Omega)\le \sum_{j=1}^k\mu_j
\end{equation}
for any $k\in \mathbb N$.
However, after that, Kr\"{o}ger found that (\ref{eq:LY}) can be improved on the special region $\Omega_{{\rm L}}$ where a bi-Lipschitz function $f:\Omega_{{\rm L}}\to ({\rm Open\ Ball})\subset \mathbb R^d$ exists.
Moreover, he also found that we can improve (\ref{eq:Kroeger}) for $k$ which is larger than the special value (see \cite{K1} for details).

Initially, these studies started from that H. Weyl \cite{W} proved that Weyl's asymptotic law implies that
\begin{equation}
\label{eq:Weyl}
\begin{array}{rc}
\mathscr{D}{\rm -Type}: & \displaystyle \mu_k\sim (2\pi)^2\left(\frac{k}{{\rm Vol}({\mathbb S}^{d-1}){\rm Vol}(\Omega)}\right)^{2/d}, \vspace{2mm}\\
\mathscr{N}{\rm -Type}: & \displaystyle \nu_{k+1}\sim (2\pi)^2\left(\frac{k}{{\rm Vol}({\mathbb S}^{d-1}){\rm Vol}(\Omega)}\right)^{2/d} 
\end{array}
\end{equation}
as $k\to \infty$.
We call the constant
\begin{align}
\label{eq:Wdk}
{\cal W}_{d,k}(\Omega):=(2\pi)^2\left(\frac{k}{{\rm Vol}({\mathbb S}^{d-1}){\rm Vol}(\Omega)}\right)^{2/d}=\left(\frac{d+2}{d}\right)\frac{1}{k}{\cal C}_{d,k}(\Omega)
\end{align}
the {\it Weyl bound} in this paper.
$(2\pi)^2{\rm Vol}(\mathbb S^{d-1})^{-2/d}$ is often called the {\it Weyl constant}.
Related to (\ref{eq:Weyl}), G. P${\rm \acute{o}}$lya \cite{P2} and others conjectured that, for any $k\in\mathbb N$,
\begin{align}
\label{eq:Polya}
\nu_{k+1}\le {\cal W}_{d,k}(\Omega)\le \mu_k
\end{align}
in {\it any} bounded region $\Omega$ with the smooth boundary.
We call this the P\'{o}lya's conjecture.
It is immediately derived that
\begin{equation}
\label{eq:Weylsum}
\begin{array}{rc}
\mathscr{D}{\rm -Type\ (Sum\ version)\ }: & \displaystyle \sum_{j=1}^k\mu_j\sim {\cal C}_{d,k}(\Omega), \vspace{2mm}\\
\mathscr{N}{\rm -Type\ (Sum\ version)\ }: & \displaystyle \sum_{j=1}^k\nu_{j}\sim {\cal C}_{d,k}(\Omega)
\end{array}
\end{equation}
as $k\to \infty$, from (\ref{eq:Weyl}).
So, Li, Yau and Kr\"{o}ger proved (\ref{eq:nCdkm}) which is another type of P${\rm \acute{o}}$lya's conjecture completely. 
This shows that their bound ${\cal C}_{d,k}(\Omega)$ for the sum of eigenvalues of $-\Delta^{\mathscr D}_{\Omega}$ or $-\Delta^{\mathscr N}_{\Omega}$ is the {\it best} in the sense of P${\rm \acute{o}}$lya (namely, in the semi-classical limit).
(\ref{eq:Polya}) has been proven affirmatively for the tiling region $\Omega$ by P${\rm \acute{o}}$lya \cite{P1}, but M. Kwa\'{s}nicki, R. S. Laugesen and B. A. Siudeja \cite{KLS} recently found that (the analogue of) P\'{o}lya's conjecture is not generally true for fractional Laplacians.
We also add that A. D. Melas \cite{Me} improved (\ref{eq:LY}) to
\begin{align}
\sum_{j=1}^k\mu_j\ge {\cal C}_{d,k}(\Omega)+M_d\frac{{\rm Vol}(\Omega)}{I(\Omega)}k,
\end{align}
where $M_d$ denotes the positive constant depending only on $d$, and 
\[
I(\Omega):=\min_{y\in \mathbb R^d}\int_{\Omega}|x-y|^2\,dx.
\]

On the one hand Li and Yau \cite{LY} also proved that
\begin{align}
\label{eq:LYbound}
\mu_{k}\ge \frac{d}{d+2}{\cal W}_{d,k}(\Omega)
\end{align}
for any $k\in \mathbb N$, on the other hand Kr\"{o}ger \cite{K2} also proved
\[
\nu_{k+1}\le {\cal K}_{d,k}(\Omega)
\]
for any $k\in \mathbb N$, where 
\begin{align}
\label{eq:Kdk}
{\cal K}_{d,k}(\Omega):=(2\pi)^2\left(\frac{d+2}{2}\right)^{2/d}\left(\frac{k}{{\rm Vol}({\mathbb S}^{d-1}){\rm Vol}(\Omega)}\right)^{2/d}
=\left(\frac{d+2}{2}\right)^{2/d}{\cal W}_{d,k}(\Omega).
\end{align}
We show that (\ref{eq:LY}) and (\ref{eq:LYbound}) can be imposed if $d\ge 3$, and we expand Li, Yau and Kr\"{o}ger's results to the case for magnetic Schr\"{o}dinger operators with no or Robin boundary conditions in this paper.

By the way, R. L. Frank, A. Laptev and S. Molchanov \cite{FLM} gave the gaps for eigenvalues $\{\lambda_n\}_{n=1}^{\infty}$ in the sense of quotients and differences of $d$-dimensional magnetic Schr\"{o}dinger operators.
It should be noted that the bounds for any of the following estimates do not depend on $A$ and $V$.
\begin{itemize}
\item The gaps in the sense of quotients:

For any $k\in \mathbb N$,
\begin{align}
\frac{\lambda_{k+1}}{\lambda_1}&\le 1+\left(\frac{d+2}{2}H_dk\right)^{2/d}, \label{eq:FLM1}\\
\frac{\lambda_{k+1}}{\lambda_1}&\le \left(1+\frac{4}{d}\right)\left(1+\frac{d}{d+2}(H_dk)^{2/d}\right), \label{eq:FLM2}
\end{align}
where $H_d$ denotes a constant given by
\[
H_d:=\frac{2d}{j_{\frac{d-2}{2}}^2J_{\frac{d}{2}}(j_{\frac{d-2}{2}})^2}.
\]
Here, $J_{s}$ denotes the $s$-order Bessel function and $j_{s}$ its first positive zero point.
\item The gap in the sense of differences:

For any $k\in \mathbb N$,
\begin{align}
\label{eq:PPW}
\lambda_{k+1}-\lambda_k\le \frac{4}{d}\left(\frac{1}{k}\sum_{j=1}^k\lambda_j\right).
\end{align}
\end{itemize}
On the one hand, (\ref{eq:FLM2}) (resp. (\ref{eq:FLM1})) is a better estimate than (\ref{eq:FLM1}) (resp. (\ref{eq:FLM2})) for large (resp. small) $k$.
On the other hand, if $A\equiv V\equiv 0$, then (\ref{eq:PPW}) is called the Payne-P\'{o}lya-Weinberger inequality.
In other words, they found that we can extend the Payne-P\'{o}lya-Weinberger inequality to the electromagnetic case.
In this paper, we also study the gaps in the sense of differences, which has a bound independent of $\sum_{j=1}^k\lambda_j$, for eigenvalues of magnetic Schr\"{o}dinger operators.
\vspace{2mm}

{\small
\begin{center}
{\bf Acknowledgement}
\end{center}
The author would like to pay tribute to three Drs. P. Li, S. -T. Yau and P. Kr\"{o}ger. 
Moreover, the author thanks Dr. Kenji Yajima so much for sharp advice.
}

\section{Case without boundary conditions}
We beforehand remark that we do not have to consider magnetic fields 
\[
B(x):=\left(\partial_j A_k-\partial_k A_j\right)_{j=1}^d,\quad 
\partial_j:=\frac{\partial}{\partial x_j}
\]
if $d=1$, by gauge transformations $G:=\exp(i\int^xA(y)\,dy)$.
(This mentioned in \cite{AHS} for the first time.
See \cite{L} for more details.)
So, let $d\ge 2$, and, $\Omega \subset {\mathbb R}^d$ be a bounded region with the smooth boundary $\partial \Omega$. 
We consider the magnetic Schr\"{o}dinger operator acting on $L^2(\Omega)$,
\begin{equation}
H_{A,V}:=\left(D_x-A(x)\right)^2+V(x),\quad D_x:=\frac{1}{i}\nabla_x,
\label{eq:HVA}
\end{equation}
which is defined by closing its quadratic form
\begin{align}
\label{eq:qHAV}
q_{H_{A,V}}[u]:=\|(\nabla-iA)u\|_{L^2(\Omega)}^2+\langle Vu,u\rangle,\quad u\in C_0^{\infty}(\Omega),
\end{align}
where $\nabla_x:=(\partial_{x_j})_{1\le j\le d}$ is the distributional gradient with respect to $x$ and 
\[
\|f\|_{L^2(\Omega)}:=\sqrt{\langle f,f \rangle},
\quad \ 
\langle f,g\rangle:=\langle f,g\rangle_{L^2(\Omega)}
=\int_{\Omega}f(x)\overline{g(x)}\,dx.
\]
$V$ is the electric scalar potential and $A$ the magnetic vector potential obeying the following Assumption \ref{assump:A0V}.

\begin{assump}
\label{assump:A0V}
We assume the followings:
\begin{itemize}
\item[i)] $A\in L^2_{{\rm loc}}(\Omega)^d:=\left\{u:\Omega \to {\mathbb R}^d:u\in L^2(K)^d\ {\rm with\ any\ compact\ subset}\ K\subset \Omega\right\}$. 
\item[ii)] $0\le V\in L^1_{{\rm loc}}(\Omega):=\left\{u:\Omega \to {\mathbb R}:u\in L^1(K)\ {\rm with\ any\ compact\ subset}\ K\subset \Omega\right\}$.
\item[iii)] $H_{A,V}$ has a compact resolvent, that is, $(H_{A,V}+1)^{-1}$ is compact in $L^2({\mathbb R}^d)$.
\end{itemize}
\end{assump}

\begin{rem}
For instance, the potentials which are called Kato-class potentials(e.g. \cite{CFKS,RS,Y}) that satisfy
\[
\lim_{r\downarrow 0}\sup_{x\in \mathbb R^d}\int_{|x-y|<r}\frac{V(y)}{|x-y|^{d-2}}\,dy=0
\]
are always $L^1_{{\rm loc}}$-valued functions.
We write ${\cal K}$ for the set of all such potentials.
Thus, ii) of Assumption \ref{assump:A0V} may be replaced with `II) $0\le V\in {\cal K}$'.
\end{rem}

It is well known(\cite{AHS,LS} and Theorem XIII.64 of \cite{RS}) that 
\begin{itemize}
\item $H_{A,V}$ is self-adjoint under the condition i) and ii) of Assumption \ref{assump:A0V}, 
\item the spectrum $\sigma(H_{A,V})$ of the self-adjoint magnetic Schr\"{o}dinger operator $H_{A,V}$ becomes a discrete subset of $\mathbb R$ under the condition iii) of Assumption \ref{assump:A0V} (also, note that then $H_{A,V}$ has no finite accumulation point), and
\item there exists the orthonomal system $\{\varphi_n\}\subset {\cal D}(H_{A,V})$ such that
\begin{align}
\label{eq:HAVlam}
H_{A,V}\varphi_n=\lambda_n\varphi_n,\quad 0<\lambda_1<\lambda_2\le \lambda_3\le \cdots \xrightarrow[n\to \infty]{} \infty
\end{align}
where ${\cal D}(T)$ is the domain of the operator $T$ and each (isolated) eigenvalue $\lambda_n$ is repeated according to multiplicities.
\end{itemize}

\subsection{Estimates for a single eigenvalue of $H_{A,V}$}
As mentioned above, we can interpret that Kr\"{o}ger \cite{K2} showed an estimate for the single eigenvalue of the free Hamiltonian $H_0:=-\Delta=-\nabla\cdot \nabla$ under the Neumann boundary condition. 
We first extend his estimate to the case for $H_{A,V}$ without boundary conditions. 
There are only a few changes, but we follows his proof basically.

The following result plays an important role.

\begin{prop}
For any $k\in \mathbb N$, the $k$-th excited state energy eigenvalue of $H_{A,V}$ holds that
\begin{align}
\lambda_{k+1}\le \inf_{B_r}
\frac{\int_{B_r\times \Omega}(|\xi-iA(y)|^2+V(y))\,d\xi\,dy-(2\pi)^d\sum_{j=1}^k\lambda_j}{{\rm Vol}(B_r){\rm Vol}(\Omega)-(2\pi)^dk}
\label{eq:EE}
\end{align}
where
\begin{align}
\label{eq:Br}
B_r:=\left\{\xi\in \Omega :|\xi|\le r,\ r^2>{\cal W}_{d,k}(\Omega)\right\}
\end{align}
and ${\cal W}_{d,k}(\Omega)$ denotes (\ref{eq:Wdk}). 
\label{prop:KLSY}
\end{prop}

\begin{proof}
Fix $k\in \mathbb N$ arbitrarily.
Let $\{\varphi_j\}_{j=1}^k\subset L^2(\Omega)$ be a set of all orthonomal eigenfunctions corresponding to eigenvalues $\lambda_1,\ldots,\lambda_k$ of $H_{A,V}$.
Remark that $\{\lambda_j\}_{j=1}^k$ obeys (\ref{eq:HAVlam}).
We consider the orthogonal projection of $h_{\xi}(y):=e^{-i\xi y}$ onto the subspace of $L^2(\Omega)$ spanned by $\{\varphi_j\}_{j=1}^k$ according to \cite{LY}:
\begin{align}
(Ph_{\xi})(\xi,y):=\sum_{j=1}^k\langle h_{\xi},\varphi_j\rangle_{L^2(\Omega)} \varphi_j(y)=\int_{\Omega}e^{-i\xi x}\Phi_k(x,y)\,dx
\label{eq:Ph}
\end{align}
where 
\begin{align}
\label{eq:Phik}
\Phi_k(x,y):=\sum_{j=1}^k\overline{\varphi_j(x)}\varphi_j(y),\quad (x,y)\in \Omega\times \Omega.
\end{align}
(\ref{eq:Ph}) can be written by the partial Fourier transform $\widehat{\Phi}_k^x$ of $\Phi_k$ with respect to the $x$-variable:
\begin{align*}
(Ph_{\xi})(\xi,y)&=(2\pi)^{d/2}\sum_{j=1}^k\overline{\widehat{\varphi}_j(\xi)}\varphi(y)=(2\pi)^{d/2}\widehat{\Phi}_k^x(\xi,y),\\ 
\widehat{u}(\xi)&:=(2\pi)^{-d/2}\int_{\Omega}u(x)e^{-i\xi x}\,dx.
\end{align*}
We also consider a function $\tilde{\varphi}_k(\xi,y):=h_{\xi}(y)-(Ph_{\xi})(\xi,y)$ according to \cite{K2}.
Putting $Q_y:=D_y-A(y)$, we obtain
\begin{align}
\lambda_{k+1}\le \inf_{B_r}\frac{\int_{B_r\times \Omega}(|Q_y\tilde{\varphi}_k(\xi,y)|^2+V(y)|\tilde{\varphi}_k(\xi,y)|^2)\,d\xi\,dy}
{\int_{B_r\times \Omega}|\tilde{\varphi}_k(\xi,y)|^2\,d\xi\,dy}
\label{eq:RRE}
\end{align}
from the mini-max principle \cite{RS}.
Then, the numerator of the fraction in the right-hand side of (\ref{eq:RRE}) is rewritten as
\begin{equation}
\begin{split}
\lefteqn{\int_{B_r\times \Omega}(|Q_y\tilde{\varphi}_k(\xi,y)|^2+V(y)|\tilde{\varphi}_k(\xi,y)|^2)\,d\xi\,dy}\qquad \\
 &=\int_{B_r\times \Omega}(|Q_yh_{\xi}(y)|^2+V(y)|h_{\xi}(y)|^2)\,d\xi\,dy \\
 &\quad -2\Re\int_{B_r\times \Omega}Q_y\tilde{\varphi}_k(\xi,y)\cdot \overline{Q_y(Ph_{\xi})(\xi,y)}\,d\xi\,dy \\
 &\quad -2\Re\int_{B_r\times \Omega}V(y)^{1/2}\tilde{\varphi}_k(\xi,y)\cdot \overline{V(y)^{1/2}(Ph_{\xi})(\xi,y)}\,d\xi\,dy \\
 &\quad -\int_{B_r\times \Omega}(|Q_y(Ph_{\xi})(\xi,y)|^2+V(y)|(Ph_{\xi})(\xi,y)|^2)\,d\xi\,dy 
\end{split}
\label{eq:nume}
\end{equation}
where $\Re z$ is the real part of $z\in \mathbb C$.

The first term in the right-hand side of (\ref{eq:nume}) is rewritten as
\begin{align*}
\int_{B_r\times \Omega}(|Q_yh_{\xi}(y)|^2+V(y)|h_{\xi}(y)|^2)\,d\xi\,dy
=\int_{B_r\times \Omega}(|\xi-iA(y)|^2+V(y))\,d\xi\,dy.
\end{align*}

The second and third terms in the right-hand side of (\ref{eq:nume}) vanish, since, for any $j=1,\ldots,k$, 
\begin{align}
\lefteqn{\int_{B_r\times \Omega}Q_y\tilde{\varphi}_k(\xi,y)\cdot \overline{Q_y\overline{\widehat{\varphi}_j(\xi)}\varphi_j(y)}\,d\xi\,dy}\hspace{10mm} \nonumber\\
 &\hspace{15mm}+\int_{B_r\times \Omega}V(y)^{1/2}\tilde{\varphi}_k(\xi,y)\cdot \overline{V(y)^{1/2}\overline{\widehat{\varphi}_j(\xi)}\varphi_j(y)}\,d\xi\,dy \nonumber\\
 &=\int_{B_r\times \Omega}\tilde{\varphi}_k(\xi,y)\widehat{\varphi}_j(\xi)\cdot \overline{H_{A(y),V(y)}\varphi_j(y)}\,d\xi\,dy \label{eq:impeq}\\
 &=\lambda_j\int_{B_r}\left(\int_{\Omega}\tilde{\varphi}_k(\xi,y)\overline{\varphi_j(y)}\,dy\right)\widehat{\varphi}_j(\xi)\,d\xi=0 \nonumber
\end{align}
from integration by parts with respect to the $y$-variable.
Here, it has been used for the last term of (\ref{eq:impeq}) that $\tilde{\varphi}_k$ is perpendicular to $Ph_{\xi}$ (that is, to every $\varphi_j$).

The fourth term in the right-hand side of (\ref{eq:nume}) is rewritten as
\begin{align}
\lefteqn{\int_{B_r\times \Omega}(|Q_y(Ph_{\xi})(\xi,y)|^2+V(y)|(Ph_{\xi})(\xi,y)|^2)\,d\xi\,dy}\hspace{10mm} \nonumber\\
 &=(2\pi)^d\int_{B_r\times \Omega}(|Q_y\widehat{\Phi}_k^x(\xi,y)|^2+V(y)|\widehat{\Phi}_k^x(\xi,y)|^2)\,d\xi\,dy \nonumber\\
 &=-(2\pi)^d\int_{B_r\times \Omega}\widehat{\Phi}_k^x(\xi,y)\cdot \overline{H_{A(y),V(y)}\widehat{\Phi}_k^x(\xi,y)}\,d\xi\,dy \label{eq:siglam}\\
 &=(2\pi)^d\int_{B_r\times \Omega}\left(\sum_{j=1}^k\lambda_j|\widehat{\varphi}_j(\xi)|^2|\varphi_j(y)|^2\right)\,d\xi\,dy \nonumber\\
 &=(2\pi)^d\sum_{j=1}^k\lambda_j\int_{B_r}|\widehat{\varphi}_j(\xi)|^2\,d\xi \nonumber
\end{align}
since $\{\varphi_j\}_{j=1}^k$ is orthonomal on $\Omega$.

We finally consider the denominator of the fraction in the right-hand side of (\ref{eq:RRE}).
But Kr\"{o}ger \cite{K2} has already derived
\begin{align*}
\int_{B_r\times \Omega}|\tilde{\varphi}_k(\xi,y)|^2\,d\xi\,dy={\rm Vol}(B_r){\rm Vol}(\Omega)-(2\pi)^d\sum_{j=1}^k\int_{B_r}|\widehat{\varphi}_j(\xi)|^2\,d\xi
\end{align*}
by using Pythagorean theorem and the orthonomality of $\{\varphi_j\}_{j=1}^k$.

We denote $H(\xi,y):=\left|\xi-iA(y)\right|^2+V(y)$ for simplicity. 
From the above, we obtain
\begin{align}
\label{eq:lamk+1}
\lambda_{k+1}\le \inf_{B_r}\frac{\int_{B_r\times \Omega}H(\xi,y)\,d\xi\,dy-(2\pi)^d\sum_{j=1}^k\lambda_j\int_{B_r}|\widehat{\varphi}_j(\xi)|^2\,d\xi}{{\rm Vol}(B_r){\rm Vol}(\Omega)-(2\pi)^d\sum_{j=1}^k\int_{B_r}|\widehat{\varphi}_j(\xi)|^2\,d\xi}
\end{align}
for any $k\in \mathbb N$.
Hence, (\ref{eq:lamk+1}) implies that
\begin{align}
\begin{split}
\lambda_{k+1}
&\le \frac{\int_{B_r\times \Omega}H(\xi,y)\,d\xi\,dy-(2\pi)^d\sum_{j=1}^k\lambda_j\int_{B_r}|\widehat{\varphi}_j(\xi)|^2\,d\xi}{{\rm Vol}(B_r){\rm Vol}(\Omega)-(2\pi)^d\sum_{j=1}^k\int_{B_r}|\widehat{\varphi}_j(\xi)|^2\,d\xi} \\
&=\frac{\int_{B_r\times \Omega}H(\xi,y)\,d\xi\,dy-(2\pi)^d\sum_{j=1}^{k}\lambda_j+(2\pi)^d\sum_{j=1}^k\lambda_j(1-\int_{B_r}|\widehat{\varphi}_j(\xi)|^2\,d\xi)}{{\rm Vol}(B_r){\rm Vol}(\Omega)-(2\pi)^d\sum_{j=1}^k\int_{B_r}|\widehat{\varphi}_j(\xi)|^2\,d\xi} \\
&\le \frac{\int_{B_r\times \Omega}H(\xi,y)\,d\xi\,dy-(2\pi)^d\sum_{j=1}^{k}\lambda_j}{{\rm Vol}(B_r){\rm Vol}(\Omega)-(2\pi)^dk}.
\label{eq:EElk}
\end{split}
\end{align}
since $0\le \int_{B_r}|\widehat{\varphi}_j(\xi)|^2\,d\xi\le 1$ for any $j=1,\ldots,k$.
Thus, we can derive (\ref{eq:EE}) from (\ref{eq:EElk}), so this completes the proof.
\end{proof}

The following result is provided by choosing the radius $r$ of $B_r$ well.

\begin{thm}[{\it Upper Bounds for $H_{A,V}$ with No Boundary Conditions}{\rm }]
\label{thm:UBNBC}
For any $k\in \mathbb N$, one has
\begin{align}
\lambda_{k+1}\le \frac{d+2}{d{\rm Vol}(\Omega)}(\|A\|_{L^2(\Omega)}^2+\|V\|_{L^1(\Omega)})+{\cal K}_{d,k}(\Omega)
\label{eq:EECor}
\end{align}
where ${\cal K}_{d,k}(\Omega)$ denotes (\ref{eq:Kdk}).
\end{thm}

\begin{proof}
Since every $\lambda_j$, $1\le j\le k$, is positive, 
\[
\lambda_{k+1}\le \frac{\int_{B_r\times \Omega}(|\xi|^2+|A(y)|^2+V(y))\,d\xi\,dy}{{\rm Vol}(B_r){\rm Vol}(\Omega)-(2\pi)^dk}
\]
from (\ref{eq:EE}).
By a simple calculation, 
\[
\int_{B_r}|\xi|^2\,d\xi=\frac{d}{d+2}r^{d+2}{\rm Vol}({\mathbb S}^{d-1}),\quad \ 
{\rm Vol}(B_r)=r^d{\rm Vol}({\mathbb S}^{d-1}), 
\]
thus, 
\begin{align}
\lambda_{k+1}\le \frac{\frac{d}{d+2}r^{d+2}{\rm Vol}({\mathbb S}^{d-1}){\rm Vol}(\Omega)+r^d{\rm Vol}({\mathbb S}^{d-1})(\|A\|_{L^2(\Omega)}^2+\|V\|_{L^1(\Omega)})}{r^d{\rm Vol}({\mathbb S}^{d-1}){\rm Vol}(\Omega)-(2\pi)^dk}.
\label{eq:EECP}
\end{align}
We now define
\begin{align}
\label{eq:rd+2k}
r(l):={\cal W}_{d,k}(\Omega)^{1/2}\left(\frac{d+l}{l}\right)^{1/d}>{\cal W}_{d,k}(\Omega)^{1/2},\quad l>0
\end{align}
and substitute this $r(l)$ for $r$ in (\ref{eq:EECP}). 
Then, 
\[
\frac{\frac{d}{d+2}r^{d+2}{\rm Vol}({\mathbb S}^{d-1}){\rm Vol}(\Omega)}{r^d{\rm Vol}({\mathbb S}^{d-1}){\rm Vol}(\Omega)-(2\pi)^dk}=\frac{(d+l)^{1+2/d}}{(d+2)l^{2/d}}{\cal W}_{d,k}(\Omega).
\]
So, the function $F(l):=(l+d)^{d+2}/l^2$ satisfies
\[
F'(l)=\frac{(l+d)^{d+1}(l-2)}{l^3}
\left\{
\begin{array}{ll}
>0 & {\rm if}\ \ l>2,\vspace{1mm}\\
=0 & {\rm if}\ \ l=2, \vspace{1mm}\\
>0 & {\rm if}\ \ 0<l<2
\end{array}
\right.
\]
and has a minimum value at $l=2$.
Hence, $r(2)$ is the best radius of $B_r$ for the desired estimate.
We immediately obtain (\ref{eq:EECor}) by setting $r=r(2)$.
\end{proof}

\begin{rem}
Since the first and second terms of (\ref{eq:EECor}) do not depend on $k$,  it is the constant ${\cal K}_{d,k}(\Omega)$, i.e. (\ref{eq:Kdk}), which decides the approximate size of the gap between two adjacent (excited state energy) eigenvalues of $H_{A,V}$.
That is, it can be expected to obtain the rough approximation
\[
\lambda_{k+1}-\lambda_{k}\approx {\cal K}_{d,k}(\Omega)-{\cal K}_{d,k-1}(\Omega)
\]
for any $k\ge 2$.
See also Corollary \ref{cor:HAVgap} for more precise gaps of eigenvalues of $H_{0,V}$.
Moreover, (\ref{eq:EECor}) indicates that, unlike (\ref{eq:PPW}) and so on, it is not necessary to know all eigenvalues of the previous terms.
\end{rem}

We next show that the sum of eigenvalues or the single eigenvalue of $H_{A,V}$ is bounded from below and that the lower bounds are given by bounds like (\ref{eq:LY}) and (\ref{eq:LYbound}).
The proofs essentially obey Li and Yau \cite{LY}.
It is important that the proof does not require the argument of Rayleigh quotients.

\begin{thm}[{\it Lower Bounds for $H_{0,V}$ with No Boundary Conditions}{\rm }]
\label{thm:LB}
We write $\lambda_j^0$, $j=1,\ldots,k$, for eigenvalues of $H_{0,V}$.
For any $k\in \mathbb N$, we have
\begin{align}
\label{eq:lam0sum}
\sum_{j=1}^k\lambda_{j}^0\ge \frac{d}{d+2}k^{d/2}{\cal W}_{d,k}(\Omega),
\end{align}
in particular
\begin{align}
\label{eq:lamlow}
\lambda_{k}^0\ge \frac{d}{d+2}k^{d/2-1}{\cal W}_{d,k}(\Omega),
\end{align}
where ${\cal W}_{d,k}(\Omega)$ denotes (\ref{eq:Wdk}).
\end{thm}

To see this, we use the following lemma. 
It was originally a half statement for the estimates from above which was pointed out by L. H\"{o}rmander and which was proved by Li and Yau \cite{LY}.

\begin{lem}
\label{lem:LY}
Suppose that the function $f:\mathbb R^d\to \mathbb R$ satisfies the followings:
\begin{itemize}
\item[i)] There exist certain constants $M_1,N_1>0$ such that $M_1\le f(x)\le N_1$ for any $x\in \mathbb R^d$.\vspace{1mm}
\item[ii)] There exist certain constants $M_2,N_2>0$ such that
\[
\left\{
\begin{array}{l}
\displaystyle M_2\le  \int_{|x|<R}|x|^2f(x)\,dx\quad \ {\rm where}\ \ R=\left(\frac{d+2}{d}\frac{M_2}{M_1{\rm Vol}(\mathbb S^{d-1})}\right)^{1/(d+2)}, \vspace{2mm}\\
\displaystyle \int_{\mathbb R^d}|x|^2f(x)\,dx\le N_2.
\end{array}
\right.
\]
\end{itemize}
Then, one has
\[
C_dM_1^{2/(d+2)}M_2^{d/(d+2)}\le \int_{\mathbb R^d}f(x)\,dx\le C_dN_1^{2/(d+2)}N_2^{d/(d+2)}
\]
where 
\begin{align}
\label{eq:HLYc}
C_d:=\left(\frac{d+2}{d}\right)^{d/(d+2)}{\rm Vol}(\mathbb S^{d-1})^{2/(d+2)}.
\end{align}
\end{lem}

\begin{proof}
We prove only the estimate from below, but its proof will be done in the same way as \cite{LY}.
Define
\[
g(x):=
\left\{
\begin{array}{ll}
M_1 & {\rm if}\ \ |x|<R, \vspace{2mm}\\
0 &{\rm if}\  \ |x|\ge R \\
\end{array}
\right.
\]
where $R$ is a positive constant obeying 
\begin{align}
\label{eq:x2gM2}
\int_{\mathbb R^d}|x|^2g(x)\,dx=M_2.
\end{align}
Since $(|x|^2-R^2)(f(x)-g(x))\le 0$ if $|x|<R$, (\ref{eq:x2gM2}) and the assumption ii) imply that
\begin{align*}
R^2\int_{|x|<R}(f(x)-g(x))\,dx&\ge \int_{|x|<R}|x|^2(f(x)-g(x))\,dx \\
&=\int_{|x|<R}|x|^2f(x)\,dx-\int_{|x|<R}|x|^2g(x)\,dx\ge 0.
\end{align*}
So, we have
\begin{align}
\label{eq:ifgeig}
\int_{|x|<R}f(x)\,dx\ge \int_{|x|<R}g(x)\,dx=M_1{\rm Vol}(\mathbb S^{d-1})R^d.
\end{align}
Calculating (\ref{eq:x2gM2}), we also have
\begin{align}
\label{eq:M2Rd+2}
M_2=\frac{d}{d+2}M_1{\rm Vol}(\mathbb S^{d-1})R^{d+2}.
\end{align}
Hence, solving (\ref{eq:M2Rd+2}) for $R$ and substituting
\begin{align*}
R=\left(\frac{d+2}{d}\frac{M_2}{M_1}\frac{1}{{\rm Vol}(\mathbb S^{d-1})}\right)^{1/(d+2)}
\end{align*}
for (\ref{eq:ifgeig}),
\begin{align*}
\int_{|x|<R}f(x)\,dx\ge \left({\rm Vol}(\mathbb S^{d-1})M_1\right)^{2/(d+2)}\left(\frac{d+2}{d}M_2\right)^{d/(d+2)}.
\end{align*}
Then, we obtain the desired inequality, since $f(x)>0$  for any $x\in \mathbb R^d$.
\end{proof}

\begin{prfLB}
We use the function, (\ref{eq:Phik}), in the proof of Proposition \ref{prop:KLSY} again.
Let us apply Lemma \ref{lem:LY} to 
\begin{align}
\label{eq:fxi}
f(\xi):=\int_{\Omega}|\widehat{\Phi}_k^x(\xi,y)|^2\,dy.
\end{align}
We estimate $f$ and the integration over $\mathbb R^d$ of $|\xi|^2f$.

On the one hand, since the Schwarz inequality implies that
\begin{align*}
|\widehat{\Phi}_k^x(\xi,y)|^2
&\le (2\pi)^{-d}\left(\int_{\Omega}|e^{-i\xi x}|^2\,dx\right)\left(\int_{\Omega}|\Phi_k(x,y)|^2\,dx\right) \\
&=(2\pi)^{-d}{\rm Vol}(\Omega)\int_{\Omega}|\Phi_k(x,y)|^2\,dx
\end{align*}
and the orthonormality of eigenfunctions implies
\[
\int_{\Omega_x\times \Omega_y}|\Phi_k(x,y)|^2\,dx\,dy=k,
\]
we have
\begin{align}
\label{eq:M1est}
(0\le )\,f(\xi)\le (2\pi)^{-d}{\rm Vol}(\Omega)k.
\end{align}

On the other hand, Li and Yau \cite{LY} have already derived
\begin{align}
\label{eq:intxi2f}
\int_{\mathbb R^d}|\xi|^2f(\xi)\,d\xi=\int_{\mathbb R^d_x\times \Omega}|\nabla_y \Phi_k(x,y)|^2\,dx\,dy
\end{align}
by simple calculation.
Since we assume that $V(y)\ge 0$ for any $y\in \Omega$,
\[
\int_{\mathbb R^d_x\times \Omega}V(y)|\Phi_k(x,y)|^2\,dx\,dy\ge 0.
\]
So, we have
\begin{align}
\label{eq:M2est}
\begin{aligned}
\int_{\mathbb R^d}|\xi|^2f(\xi)\,d\xi&\le \int_{\mathbb R^d_x\times \Omega}(|\nabla_y \Phi_k(x,y)|^2+V(y)|\Phi_k(x,y)|^2)\,dx\,dy \\
&=\sum_{j=1}^k\lambda_j^0\int_{\Omega}|\varphi_j(x)|^2\,dx \\
&=\sum_{j=1}^k\lambda_j^0
\end{aligned}
\end{align}
in the same way as (\ref{eq:siglam}).

Thus, choosing that
\[
N_1:=(2\pi)^{-d}{\rm Vol}(\Omega)k,\quad 
N_2:=\sum_{j=1}^k\lambda_j^0
\]
from (\ref{eq:M1est}) and (\ref{eq:M2est}), Lemma \ref{lem:LY} implies
\begin{align}
\label{eq:LYlem1}
\int_{\mathbb R^d_{\xi}\times \Omega}|\widehat{\Phi}_k^x(\xi,y)|^2\,d\xi\,dy
\le C_d((2\pi)^{-d}{\rm Vol}(\Omega)k)^{2/(d+2)}\left(\sum_{j=1}^k\lambda_j^0\right)^{d/(d+2)}.
\end{align}
However, Plancherel's theorem and the orthonormality of $\{\varphi_j\}_{j=1}^k$ tell us that 
\begin{align}
\label{eq:LYlem2}
\int_{\mathbb R^d_{\xi}\times \Omega}|\widehat{\Phi}_k^x(\xi,y)|^2\,d\xi\,dy=\int_{\mathbb R^d_x\times \Omega}|\Phi_k(x,y)|^2\,dx\,dy\ge k.
\end{align}
By virtue of (\ref{eq:LYlem1}) and (\ref{eq:LYlem2}),
\[
\sum_{j=1}^k\lambda_j^0\ge \frac{(2\pi)^2}{C_d^{1+2/d}{\rm Vol}(\Omega)^{2/d}}k.
\]
Recall (\ref{eq:Wdk}) and (\ref{eq:HLYc}), then this completes the proof of (\ref{eq:lam0sum}).

Now, we can estimate as
\[
\sum_{j=1}^k\lambda_j^0\le k\lambda_k^0
\]
by the monotonicity of eigenvalues, so it is easy to see (\ref{eq:lamlow}) from (\ref{eq:lam0sum}).

This completes the proof of the theorem.
\qed
\end{prfLB}

\begin{rem}
The above proof is the same as the proof of the Li-Yau inequality (\ref{eq:LYbound}), but our result is improved to the same estimate as (\ref{eq:lamlow}) if $d\ge 3$, $A\equiv V\equiv 0$ and having Dirichlet boundary condition.
In fact, we gain that
\[
\left.
\begin{array}{ll}
k^{d/2-1}=1 & {\rm if}\ d=2, \vspace{2mm}\\
k^{d/2-1}\ge 1 & {\rm if}\ d\ge 3
\end{array}
\right\}.
\]
\end{rem}

\begin{cor}[{\it Gaps of eigenvalues of $H_{0,V}$}]
\label{cor:HAVgap}
We write 
\begin{align}
\label{eq:MVOmega}
{\cal M}(V;\Omega):=\frac{1}{{\rm Vol}(\Omega)}\int_{\Omega}V(y)\,dy=\frac{\|V\|_{L^1(\Omega)}}{{\rm Vol}(\Omega)}
\end{align}
for the average in the sense of integrals of $V$ over $\Omega$.
If $A\equiv 0$, we have
\begin{align}
\label{eq:HVgap1}
\lambda_{k+1}^0-\lambda_k^0\le {\cal K}_{d,k}(\Omega)-\frac{d}{d+2}k^{d/2-1}{\cal W}_{d,k+1}(\Omega)+\frac{d+2}{d}{\cal M}(V;\Omega),
\end{align}
in particular
\begin{align}
\label{eq:HVgap2}
\lambda_{k+1}^0-\lambda_1^0\le \sum_{j=1}^k\left({\cal K}_{d,k}(\Omega)-\frac{d}{d+2}k^{d/2-1}{\cal W}_{d,k+1}(\Omega)+\frac{d+2}{d}{\cal M}(V;\Omega)\right),
\end{align}
for any $k\in \mathbb N$.
Here, ${\cal K}_{d,k}(\Omega)$ and ${\cal W}_{d,k}(\Omega)$ denote (\ref{eq:Kdk}) and (\ref{eq:Wdk}) respectively.
\end{cor}

\begin{proof}
(\ref{eq:HVgap1}) is obvious from (\ref{eq:EECor}) and (\ref{eq:lamlow}).
As for (\ref{eq:HVgap2}), consider $\sum_{j=1}^k$ of both sides of (\ref{eq:HVgap1}).
\end{proof}

\begin{rem}
\begin{itemize}
\item[(1)] L. Erd\"{o}s, M. Loss and V. Vougalter \cite{ELV} have proved that if $A$ is a constant magnetic field, the Li--Yau inequality holds for $H_{A,V}$:
\begin{align*}
\sum_{j=1}^k\lambda_j&\ge \frac{d}{d+2}k{\cal W}_{d,k+1}(\Omega), \\
\lambda_j&\ge \frac{d}{d+2}{\cal W}_{d,k+1}(\Omega),\quad j\in \mathbb N.
\end{align*}
(However, it may be unclear whether that statement is true in case of the general magnetic field.)
Thus, if $A$ is a constant magnetic field, (\ref{eq:HVgap1}) in Corollary \ref{cor:HAVgap} can be rewritten as 
\[
\lambda_{k+1}-\lambda_k\le {\cal K}_{d,k}(\Omega)-\frac{d}{d+2}{\cal W}_{d,k+1}(\Omega)+\frac{d+2}{d}{\cal M}(V;\Omega).
\]
\item[(2)] Remark that, the inequality
\[
\sum_{j=1}^k\lambda_j\ge \sum_{j=1}^k\lambda_j^0
\]
holds if $k=1$, but it does not hold in general if $k=2$.
See Remark 2 of \cite{ELV} for details.
So, we cannot say that magnetic Li--Yau inequalities obviously hold from diamagnetic inequalities.
\end{itemize}
\end{rem}

\subsection{Estimates for the sum of eigenvalues of $H_{A,V}$}
We can obtain the following estimate for the sum of eigenvalues of the magnetic Schr\"{o}dinger operator with no boundary conditions.

\begin{thm}
\label{thm:ESE}
For any $k\in \mathbb N$, one has
\begin{align}
\sum_{j=1}^k\lambda_j\le {\cal C}_{d,k}(\Omega)+k^2\frac{\|A\|_{L^2(\Omega)}^2}{{\rm Vol}(\Omega)}+k\frac{\|V\|_{L^1(\Omega)}}{{\rm Vol}(\Omega)}
\label{eq:ESE}
\end{align}
where ${\cal C}_{d,k}(\Omega)$ denotes (\ref{eq:Cdk}).
\end{thm}

\begin{proof}
We set the suitable radius $r$ of the ball $B_r$, (\ref{eq:Br}), to
\begin{align}
\label{eq:rk+1}
r=2\pi\left(\frac{k+1}{{\rm Vol}({\mathbb S}^{d-1}){\rm Vol}(\Omega)}\right)^{1/d}={\cal W}_{d,k+1}(\Omega)^{1/2}>{\cal W}_{d,k}(\Omega)^{1/2}.
\end{align}
Moreover, since $\lambda_j\le \lambda_{k+1}$ for any $1\le j\le k$, (\ref{eq:EECP}) implies
\begin{align}
\label{eq:lamj}
\lambda_{j}\le \frac{\frac{d}{d+2}r^{d+2}{\rm Vol}({\mathbb S}^{d-1}){\rm Vol}(\Omega)+r^d{\rm Vol}({\mathbb S}^{d-1})(\|A\|_{L^2(\Omega)}^2+\|V\|_{L^1(\Omega)})}{r^d{\rm Vol}({\mathbb S}^{d-1}){\rm Vol}(\Omega)-(2\pi)^dk}
\end{align}
for any $1\le j\le k+1$.
Then, substituing (\ref{eq:rk+1}) in (\ref{eq:lamj}), we have
\begin{align*}
\sum_{j=1}^{k+1}\lambda_j\le \frac{d}{d+2}(2\pi)^2({\rm Vol}(\mathbb S^{d-1}){\rm Vol}(\Omega))^{-2/d}(k+1)^{1+2/d}+\frac{(k+1)^2}{{\rm Vol}(\Omega)}\|A\|_{L^2(\Omega)}^2+\frac{k+1}{{\rm Vol}(\Omega)}\|V\|_{L^1(\Omega)}.
\end{align*}
Hence, this completes the proof.
\end{proof}

\begin{rem}
If $A\neq 0$, Theorem \ref{thm:ESE} shows
\[
\sum_{j=1}^k\lambda_j\le {\cal C}_{d,k}(\Omega)+k{\cal M}(V;\Omega).
\]
This means that the integral mean value of $V$ over $\Omega$, (\ref{eq:MVOmega}), multiplied by the number of eigenvalues is added to ${\cal C}_{d,k}(\Omega)$.
\end{rem}

\section{Case with Robin boundary conditions}
We hereafter assume the following. 
$\Omega\subset \mathbb R^d$ has still the smooth boundary.

\begin{assump}[{\it Robin Boundary Conditions}]
\label{assump:RBC}
$\sigma\in L^{\infty}(\partial \Omega)$ and
\begin{equation}
(\nabla-iA)u\cdot {\bf n}=-\sigma u\quad {\rm on}\ \partial \Omega.
\end{equation}
Here, recall that ${\bf n}$ is the outer normal vector on $\partial \Omega$. 
\end{assump}

\begin{rem}
Notice that Robin boundary conditions become Neumann boundary conditions (resp. Dirichlet boundary conditions) if $\sigma\equiv 0$ (resp. if $\sigma(x)\to 0$ as $|x|\to \infty$).
\end{rem}

In this section, we consider the magnetic Schr\"{o}dinger operator acting on $L^2(\Omega)$ with Robin boundary condition:
\begin{equation}
H_{A,V}^{\mathscr{R}}:=(D_x-A)^2+V
\end{equation}
defined by closing its quadratic form
\begin{equation}
q_{H_{A,V}^{\mathscr{R}}}[u]:=\|(\nabla-iA)u\|_{L^2(\Omega)}^2+\langle Vu,u\rangle+\int_{\partial \Omega}\sigma(x)|u(x)|^2\,dS
\label{eq:qHVAR}
\end{equation}
for $u\in {\cal Q}(q_{H_{A,V}^{\mathscr{R}}})$, where $dS$ denotes the surface measure and 
\[
{\cal Q}(q_{H_{A,V}^{\mathscr{R}}}):=\left\{u\in L^2(\Omega):(\nabla-iA)u\in L^2(\Omega)^d,\,V^{1/2}u\in L^2(\Omega)\right\}
\]
the form-domain of $q_{H_{A,V}^{\mathscr{R}}}$.
Moreover, let us think that $H_{A,V}^{\mathscr{R}}$ satisfies Assumption \ref{assump:A0V} and Assumption \ref{assump:RBC}.
Then, we suppose that $H_{A,V}^{\mathscr{R}}$ has eigenvalues $\lambda^{\mathscr{R}}_j$, $1\le j\le k$.
However, we must remark that $H_{A,V}$ with Robin boundary condition may have negative eigenvalues if $\sigma<0$, from (\ref{eq:qHVAR}) and the Rayleigh-Ritz quotient.
The negative eigenvalues will appear under the influence of only $\sigma$, and, $V$ works to reduce the number of the negative eigenvalues since $V\ge 0$.
The biggest difference with Dirichlet boundary conditions and Neumann boundary conditions of Robin boundary conditions is that the negative eigenvalues may appear, so Robin boundary conditions when $\sigma<0$ are sometimes called {\it Steklov boundary conditions} (specifically \cite{GA}).

However, we consistently investigate the case that $H_{A,V}^{\mathscr{R}}$ has positive eigenvalues by assuming that $V$ is large enough.

\subsection{Estimates for eigenvalues of $H_{A,V}^{\mathscr{R}}$}
Hereafter, we write $\|\cdot\|_{S}=\|\cdot\|_{L^2(S)}$ for simplicity.
Recall the notation $Q:=D-A$.
The mini-max principle implies that
\begin{align}
\lambda^{\mathscr{R}}_{k+1}\le \inf_{B_r}\frac{\|Q_y\tilde{\varphi}_k\|_{B_r\times \Omega}^2+\|V^{1/2}\tilde{\varphi}_k\|_{B_r\times \Omega}^2+\int_{\partial \Omega}\sigma(y)\|\tilde{\varphi}_k(\cdot,y)\|_{B_r}^2\,dS}
{\|\tilde{\varphi}_k\|_{B_r\times \Omega}^2}
\label{eq:Rmmp}
\end{align}
for any $k\in \mathbb N$.
Like Proposition \ref{prop:KLSY}, let us deform the molecule of the fraction in the right-hand side of (\ref{eq:Rmmp}).
We have, in view of (\ref{eq:qHVAR}), that
\begin{align}
&\langle u,H_{A,V}^{\mathscr{R}}v\rangle_{L^2(\Omega)} \nonumber\\
&=\langle Qu,Qv\rangle_{L^2(\Omega)}+\langle V^{1/2}u,V^{1/2}v\rangle_{L^2(\Omega)}-\int_{\partial \Omega}u\,\overline{(\nabla-iA)v\cdot {\bf n}}\,dS \label{eq:uHAVRv}\\
&=\langle Qu,Qv\rangle_{L^2(\Omega)}+\langle V^{1/2}u,V^{1/2}v\rangle_{L^2(\Omega)}+\int_{\partial \Omega}\sigma(x)u(x)\overline{v(x)}\,dS \nonumber
\end{align}
for $u,v\in L^2(\Omega)$. 
Thus, putting $u=\widehat{\varphi}_j(\xi)\tilde{\varphi}_k(\xi,y)$ and $v=\varphi_j(y)$ in (\ref{eq:uHAVRv}), the important equation (\ref{eq:impeq}) in the proof of Proposition \ref{prop:KLSY} corresponds to
\begin{align*}
&\int_{B_r}\left(\big\langle Q_y\widehat{\varphi}_j(\xi)\tilde{\varphi}_k(\xi,y),Q_y\varphi_j(y)\big\rangle_{L^2(\Omega)}+\big\langle V(y)\widehat{\varphi}_j(\xi)\tilde{\varphi}_k(\xi,y),\varphi_j(y)\big\rangle_{L^2(\Omega)}\right)d\xi \\
&\hspace{50mm} +\int_{\partial \Omega}\sigma(y)\widehat{\varphi}_j(\xi)\tilde{\varphi}_k(\xi,y)\overline{\varphi_j(y)}\,dS \\
&=\big\langle \widehat{\varphi}_j(\xi)\tilde{\varphi}_k(\xi,y),H_{A,V}^{\mathscr{R}}\varphi_j(y)\big\rangle_{L^2(\Omega)} \\
&=\lambda^{\mathscr{R}}_j\int_{B_r}\widehat{\varphi}_j(\xi)\big\langle \tilde{\varphi}_k(\xi,y),\varphi_j(y)\big\rangle_{L^2(\Omega)}\,d\xi=0,
\end{align*}
where $\lambda^{\mathscr{R}}_j$ denotes the $j$-th eigenvalue of $H_{A,V}^{\mathscr{R}}$.

\subsubsection{In case $\sigma$ is a positive valued function}
Let $\sigma|_{\partial \Omega}>0$.
We write $\lambda^{\mathscr{R}_+}_j$ for $j$-th eigenvalue of $H_{A,V}^{\mathscr{R}}$ with $\sigma|_{\partial \Omega}>0$.
We should estimate the third term of the molecule in (\ref{eq:Rmmp}) from above.
We can in fact estimate it as follows:
\begin{align*}
\int_{\partial \Omega}\sigma(y)\|\tilde{\varphi}_k(\cdot,y)\|_{B_r}^2\,dS&=\int_{\partial \Omega}\sigma(y)\|h_{\xi}-Ph_{\xi}\|_{B_r}^2\,dS \\
&=\int_{\partial \Omega}\sigma(y)(\|h_{\xi}\|_{B_r}^2-\|Ph_{\xi}\|_{B_r}^2)\,dS \\
&\le {\rm Vol}(B_r)\int_{\partial \Omega}\sigma(y)\,dS \\
&\le {\rm Vol}(B_r){\rm Ar}(\Omega)\|\sigma\|_{L^{\infty}(\partial \Omega)} \\
&=r^d{\rm Vol}(\mathbb S^{d-1}){\rm Ar}(\Omega)\|\sigma\|_{L^{\infty}(\partial \Omega)},
\end{align*}
since $\sigma|_{\partial \Omega}>0$ and $\sigma\in L^{\infty}(\partial \Omega)$.
Here, ${\rm Ar}(\Omega)$ denotes the surface area of $\Omega$.

Therefore, Proposition \ref{prop:KLSY} holds in Robin boundary case too, that is,

\begin{prop}
\label{prop:EEHR}
For any $k\in \mathbb N$, the $k$-th excited state energy eigenvalue of $H_{A,V}^{\mathscr{R}}$ with $\sigma|_{\partial \Omega}>0$ holds that
\begin{align}
\lambda_{k+1}^{\mathscr{R}_+}\le \inf_{B_r}
\frac{\int_{B_r\times \Omega}(|\xi-iA(y)|^2+V(y))\,d\xi\,dy+\|\sigma\|_{\partial \Omega}{\rm Vol}(B_r){\rm Ar}(\Omega)-(2\pi)^d\sum_{j=1}^k\lambda_j}{{\rm Vol}(B_r){\rm Vol}(\Omega)-(2\pi)^dk},
\end{align}
where $\|\sigma\|_{\partial \Omega}=\|\sigma\|_{L^{\infty}(\partial \Omega)}$ and $B_r$ denotes (\ref{eq:Br}).
\end{prop}

Hereby, we can obtain the following results in the same way as Theorem \ref{thm:UBNBC}, Theorem \ref{thm:LB}, Corollary \ref{cor:HAVgap} and Theorem \ref{thm:ESE}.

\begin{thm}[{\it Upper Bounds for $H_{A,V}^{\mathscr{R}}$ with $\sigma|_{\partial \Omega}>0$}]
For any $k\in \mathbb N$, one has
\begin{align}
\lambda_{k+1}^{\mathscr{R}_+}\le \frac{d+2}{d{\rm Vol}(\Omega)}\left(\|A\|_{L^2(\Omega)}^2+\|V\|_{L^1(\Omega)}+d{\rm Ar}(\Omega)\|\sigma\|_{L^{\infty}(\partial \Omega)}\right)+{\cal K}_{d,k}(\Omega)
\label{eq:EEHRCor}
\end{align}
where ${\cal K}_{d,k}(\Omega)$ denotes (\ref{eq:Kdk}).
\end{thm}

\begin{proof}
Recall (\ref{eq:rd+2k}) and choose $r=r(2)$.
We leave this detailed calculation to the reader.
\end{proof}

\begin{thm}[{\it Lower Bounds for $H_{0,V}^{\mathscr{R}}$ with $\sigma|_{\partial \Omega}>0$}]
We write $\lambda_j^{\mathscr{R}_+,0}$, $j=1,\ldots,k$, for eigenvalues of $H_{0,V}^{\mathscr{R}_+}$.
For any $k\in \mathbb N$, we have
\begin{align*}
\sum_{j=1}^k\lambda_{j}^{\mathscr{R}_+,0}\ge \frac{d}{d+2}k^{d/2}{\cal W}_{d,k}(\Omega),
\end{align*}
in particular
\begin{align*}
\lambda_{k}^{\mathscr{R}_+,0}\ge \frac{d}{d+2}k^{d/2-1}{\cal W}_{d,k}(\Omega),
\end{align*}
where ${\cal W}_{d,k}(\Omega)$ denotes (\ref{eq:Wdk}).
\end{thm}

\begin{proof}
Since $\int_{\partial \Omega}\sigma(y)\|\tilde{\varphi}_k(\cdot,y)\|_{B_r}^2\,dS\ge 0$, we can estimate as follows:
\begin{align*}
&\int_{\mathbb R^d_{\xi}\times \Omega}|\xi|^2|\widehat{\Phi}_k^x(\xi,y)|^2\,d\xi\,dy \\
&\le \int_{\mathbb R^d\times \Omega}(|\nabla_y\Phi_k(x,y)|^2+V(y)|\Phi_k(x,y)|^2)\,dx\,dy+\int_{\partial \Omega}\sigma(y)\|\tilde{\varphi}_k(\cdot,y)\|_{B_r}^2\,dS \\
&=\sum_{j=1}^k\lambda_{j}^{\mathscr{R}_+,0}.
\end{align*}
So, the proof of this theorem is obvious.
\end{proof}

\begin{cor}[{\it Gaps of eigenvalues of $H_{0,V}^{\mathscr{R}}$ with $\sigma|_{\partial \Omega}>0$}]
\label{cor:HAVRgap}
Recall the notation of (\ref{eq:MVOmega}).
If $A\equiv 0$, we have
\begin{align}
\label{eq:lk+1R}
\lambda_{k+1}^{\mathscr{R}_+,0}-\lambda_k^{\mathscr{R}_+,0}\le {\cal K}_{d,k}(\Omega)-\frac{d}{d+2}k^{d/2-1}{\cal W}_{d,k+1}(\Omega)+\frac{d+2}{d}{\cal M}(V;\Omega)+(d+2)\frac{{\rm Ar}(\Omega)}{{\rm Vol}(\Omega)}\|\sigma\|_{L^{\infty}(\partial \Omega)},
\end{align}
in particular
\begin{align*}
\lambda_{k+1}^{\mathscr{R}_+,0}-\lambda_1^{\mathscr{R}_+,0}\le \sum_{j=1}^k\left({\cal K}_{d,k}(\Omega)-\frac{d}{d+2}k^{d/2-1}{\cal W}_{d,k+1}(\Omega)+\frac{d+2}{d}{\cal M}(V;\Omega)+(d+2)\frac{{\rm Ar}(\Omega)}{{\rm Vol}(\Omega)}\|\sigma\|_{L^{\infty}(\partial \Omega)}\right)
\end{align*}
for any $k\in \mathbb N$.
Here, ${\cal K}_{d,k}(\Omega)$ and ${\cal W}_{d,k}(\Omega)$ denote (\ref{eq:Kdk}) and (\ref{eq:Wdk}) respectively.
\end{cor}

\begin{thm}[{\it Lower Bounds for $H_{A,V}^{\mathscr{R}}$ with $\sigma|_{\partial \Omega}>0$}]
For any $k\in \mathbb N$, we have
\begin{align*}
\sum_{j=1}^k\lambda_{j}^{\mathscr{R}_+}\ge \frac{d}{d+2}k{\cal W}_{d,k}(\Omega),
\end{align*}
in particular
\begin{align*}
\lambda_{k}^{\mathscr{R}_+}\ge \frac{d}{d+2}{\cal W}_{d,k}(\Omega),
\end{align*}
where ${\cal W}_{d,k}(\Omega)$ denotes (\ref{eq:Wdk}).
\end{thm}

\begin{cor}[{\it Gaps of eigenvalues of $H_{A,V}^{\mathscr{R}}$ with $\sigma|_{\partial \Omega}>0$}]
Recall the notation of (\ref{eq:MVOmega}).
If $A$ is a constant magnetic field, we have
\begin{align}
\begin{aligned}
\lambda_{k+1}^{\mathscr{R}_+}-\lambda_k^{\mathscr{R}_+}\le {\cal K}_{d,k}(\Omega)-\frac{d}{d+2}{\cal W}_{d,k+1}(\Omega)+\frac{d+2}{d}{\cal M}(V;\Omega)+(d+2)\frac{{\rm Ar}(\Omega)}{{\rm Vol}(\Omega)}\|\sigma\|_{L^{\infty}(\partial \Omega)},
\end{aligned}
\end{align}
in particular
\begin{align*}
\lambda_{k+1}^{\mathscr{R}_+}-\lambda_1^{\mathscr{R}_+}\le \sum_{j=1}^k\left({\cal K}_{d,k}(\Omega)-\frac{d}{d+2}{\cal W}_{d,k+1}(\Omega)+\frac{d+2}{d}{\cal M}(V;\Omega)+(d+2)\frac{{\rm Ar}(\Omega)}{{\rm Vol}(\Omega)}\|\sigma\|_{L^{\infty}(\partial \Omega)}\right)
\end{align*}
for any $k\in \mathbb N$.
Here, ${\cal K}_{d,k}(\Omega)$ and ${\cal W}_{d,k}(\Omega)$ denote (\ref{eq:Kdk}) and (\ref{eq:Wdk}) respectively.
\end{cor}

\begin{thm}
For any $k\in \mathbb N$, one has
\begin{align}
\label{eq:slR}
\sum_{j=1}^k\lambda_j^{\mathscr{R}_+}\le {\cal C}_{d,k}(\Omega)+k^2\frac{{\rm Ar}(\Omega)}{{\rm Vol}(\Omega)}\|\sigma\|_{L^{\infty}(\partial \Omega)}+k^2\frac{\|A\|_{L^2(\Omega)}^2}{{\rm Vol}(\Omega)}+k\frac{\|V\|_{L^1(\Omega)}}{{\rm Vol}(\Omega)}
\end{align}
where ${\cal C}_{d,k}(\Omega)$ denotes (\ref{eq:Cdk}).
\end{thm}

\begin{proof}
Choose $r$ in the same way as (\ref{eq:rk+1}).
We leave this detailed calculation to the reader.
\end{proof}

\begin{rem}
A constant ${\rm Ar}(\Omega)/{\rm Vol}(\Omega)$ appearing in (\ref{eq:EEHRCor}), (\ref{eq:lk+1R}) and (\ref{eq:slR}) is the specific surface area of $\Omega$.
\end{rem}

\subsubsection{In case $\sigma$ is a negative valued function and all eigenvalues are positive}
Let $\sigma|_{\partial \Omega}<0$.
We write $\lambda^{\mathscr{R}_-}_j$ for the $j$-th eigenvalue of $H_{A,V}^{\mathscr{R}}$ with $\sigma|_{\partial \Omega}<0$.
We suppose
\[
0<\lambda_1^{\mathscr{R}_-}<\lambda_2^{\mathscr{R}_-}\le \lambda_3^{\mathscr{R}_-}\le \cdots.
\]
Then, (\ref{eq:Rmmp}) is rewritten as
\[
\lambda^{\mathscr{R}_-}_{k+1}\le \inf_{B_r}\frac{\|Q_y\tilde{\varphi}_k\|_{B_r\times \Omega}^2+\|V^{1/2}\tilde{\varphi}_k\|_{B_r\times \Omega}^2}
{\|\tilde{\varphi}_k\|_{B_r\times \Omega}^2},
\]
but this is equal to (\ref{eq:RRE}).
So, we have exactly the same results as Proposition \ref{prop:KLSY}, Theorem \ref{thm:UBNBC} and Theorem \ref{thm:ESE} in this case.
Theorem \ref{thm:LB} for $H_{A,V}^{\mathscr{R}}$ with $\sigma|_{\partial \Omega}<0$ also holds, because the proof does not depend on Rayleigh quotients.
Hence, we can, in addition, obtain Corollary \ref{cor:HAVgap} for $H_{A,V}^{\mathscr{R}}$ with $\sigma|_{\partial \Omega}<0$.

\begin{prop}
For any $k\in \mathbb N$, the $k$-th excited state energy eigenvalue of $H_{A,V}^{\mathscr{R}}$ with $\sigma|_{\partial \Omega}<0$ holds that
\begin{align*}
\lambda_{k+1}^{\mathscr{R}_-}\le \inf_{B_r}
\frac{\int_{B_r\times \Omega}(|\xi-iA(y)|^2+V(y))\,d\xi\,dy-(2\pi)^d\sum_{j=1}^k\lambda_j}{{\rm Vol}(B_r){\rm Vol}(\Omega)-(2\pi)^dk},
\end{align*}
where $B_r$ denotes (\ref{eq:Br}).
\end{prop}

\begin{thm}[{\it Upper Bounds for $H_{A,V}^{\mathscr{R}}$ with $\sigma|_{\partial \Omega}<0$}]
For any $k\in \mathbb N$, one has
\begin{align*}
\lambda_{k+1}^{\mathscr{R}_-}\le \frac{d+2}{d{\rm Vol}(\Omega)}\left(\|A\|_{L^2(\Omega)}^2+\|V\|_{L^1(\Omega)}\right)+{\cal K}_{d,k}(\Omega)
\end{align*}
where ${\cal K}_{d,k}(\Omega)$ denotes (\ref{eq:Kdk}).
\end{thm}

\begin{thm}[{\it Lower Bounds for $H_{0,V}^{\mathscr{R}}$ with $\sigma|_{\partial \Omega}<0$}]
We write $\lambda_j^{\mathscr{R}_-,0}$, $j=1,\ldots,k$, for eigenvalues of $H_{0,V}^{\mathscr{R}}$ with $\sigma|_{\partial \Omega}<0$.
For any $k\in \mathbb N$, we have
\begin{align*}
\sum_{j=1}^k\lambda_{j}^{\mathscr{R}_-,0}\ge \frac{d}{d+2}k^{d/2}{\cal W}_{d,k}(\Omega),
\end{align*}
in particular
\begin{align*}
\lambda_{k}^{\mathscr{R}_-,0}\ge \frac{d}{d+2}k^{d/2-1}{\cal W}_{d,k}(\Omega),
\end{align*}
where ${\cal W}_{d,k}(\Omega)$ denotes (\ref{eq:Wdk}).
\end{thm}

\begin{cor}[{\it Gaps of eigenvalues of $H_{0,V}^{\mathscr{R}}$ with $\sigma|_{\partial \Omega}<0$}]
\label{cor:HAVR-gap}
Recall the notation of (\ref{eq:MVOmega}).
If $A\equiv 0$, we have
\begin{align*}
\lambda_{k+1}^{\mathscr{R}_-,0}-\lambda_k^{\mathscr{R}_-,0}
\le {\cal K}_{d,k}(\Omega)-\frac{d}{d+2}k^{d/2-1}{\cal W}_{d,k+1}(\Omega)+\frac{d+2}{d}{\cal M}(V;\Omega),
\end{align*}
in particular
\begin{align*}
\lambda_{k+1}^{\mathscr{R}_-,0}-\lambda_1^{\mathscr{R}_-,0}
\le 
\sum_{j=1}^k\left({\cal K}_{d,k}(\Omega)-\frac{d}{d+2}k^{d/2-1}{\cal W}_{d,k+1}(\Omega)+\frac{d+2}{d}{\cal M}(V;\Omega)\right)
\end{align*}
for any $k\in \mathbb N$.
Here, ${\cal K}_{d,k}(\Omega)$ and ${\cal W}_{d,k}(\Omega)$ denote (\ref{eq:Kdk}) and (\ref{eq:Wdk}) respectively.
\end{cor}

\begin{thm}[{\it Lower Bounds for $H_{A,V}^{\mathscr{R}}$ with $\sigma|_{\partial \Omega}<0$}]
For any $k\in \mathbb N$, we have
\begin{align*}
\sum_{j=1}^k\lambda_{j}^{\mathscr{R}_-}\ge \frac{d}{d+2}k{\cal W}_{d,k}(\Omega),
\end{align*}
in particular
\begin{align*}
\lambda_{k}^{\mathscr{R}_-}\ge \frac{d}{d+2}{\cal W}_{d,k}(\Omega),
\end{align*}
where ${\cal W}_{d,k}(\Omega)$ denotes (\ref{eq:Wdk}).
\end{thm}

\begin{cor}[{\it Gaps of eigenvalues of $H_{A,V}^{\mathscr{R}}$ with $\sigma|_{\partial \Omega}<0$}]
\label{cor:HAVR-gap}
Recall the notation of (\ref{eq:MVOmega}).
If $A$ is a constant magnetic field, we have
\begin{align*}
\lambda_{k+1}^{\mathscr{R}_-}-\lambda_k^{\mathscr{R}_-}
\le {\cal K}_{d,k}(\Omega)-\frac{d}{d+2}{\cal W}_{d,k+1}(\Omega)+\frac{d+2}{d}{\cal M}(V;\Omega),
\end{align*}
in particular
\begin{align*}
\lambda_{k+1}^{\mathscr{R}_-}-\lambda_1^{\mathscr{R}_-}
\le 
\sum_{j=1}^k\left({\cal K}_{d,k}(\Omega)-\frac{d}{d+2}{\cal W}_{d,k+1}(\Omega)+\frac{d+2}{d}{\cal M}(V;\Omega)\right)
\end{align*}
for any $k\in \mathbb N$.
Here, ${\cal K}_{d,k}(\Omega)$ and ${\cal W}_{d,k}(\Omega)$ denote (\ref{eq:Kdk}) and (\ref{eq:Wdk}) respectively.
\end{cor}

\begin{thm}
For any $k\in \mathbb N$, one has
\begin{align*}
\sum_{j=1}^k\lambda_j^{\mathscr{R}_-}\le {\cal C}_{d,k}(\Omega)+k^2\frac{\|A\|_{L^2(\Omega)}^2}{{\rm Vol}(\Omega)}+k\frac{\|V\|_{L^1(\Omega)}}{{\rm Vol}(\Omega)}
\end{align*}
where ${\cal C}_{d,k}(\Omega)$ denotes (\ref{eq:Cdk}).
\end{thm}
\vspace{4mm}

\begin{center}
{\sc Comments}
\end{center}

Our `homework' is the study of the estimates for negative eigenvalues of magnetic Schr\"{o}dinger operators with Robin boundary conditions.
The author wants to mention that on another occasion.

\end{document}